\documentclass[11pt]{article}
 \usepackage{amsfonts,amssymb,amsthm,amsmath,cite,bbm}
 \usepackage{mathrsfs}
  	
\newcommand{\N}{\mathbb N} 
\newcommand{\R}{\mathbb R} 
\newcommand{\Rn}{\R^n}
\newcommand{\Sph}{\mathbb{S}^{n-1}}

\newcommand{\elim}{\operatorname{epi-lim}}
\newcommand{\hlim}{\operatorname{hypo-lim}}
\newcommand{\vm}{\operatorname{m}} 
\newcommand{\K}{{\mathcal K}}
\newcommand{\Kn}{\K^n} 
\newcommand{\Ko}{\Kn_{o}}
\renewcommand{\P}{{\mathcal P}} 
\newcommand{\Pn}{\P^n}
\newcommand{\Po}{\Pn_{o}}

\newcommand{\CV}{\operatorname{Conv}(\Rn)}
\newcommand{\LC}{\operatorname{LC}(\Rn)}

\renewcommand{\d}{\,\mathrm{d}}

\newcommand{\mx}{\mathbin{\vee}} 
\newcommand{\mn}{\mathbin{\wedge}}

\newcommand{\conv}{\operatorname{conv}}

\newcommand{\dom}{\operatorname{dom}}
\newcommand{\epi}{\operatorname{epi}}
\newcommand{\pos}{\operatorname{pos}}

\newcommand{\Ind}{\mathrm{I}}
\newcommand{\Char}{\operatorname{\chi}}
\renewcommand{\l}{\ell}

\newcommand{\oD}{\operatorname{D}}
\newcommand{\oY}{\operatorname{Y}}
\newcommand{\oZ}{\operatorname{Z}}
\newcommand{\oM}{\operatorname{M}}

\newcommand\SLn{\operatorname{SL}(n)}

\newcommand{\eto}{\stackrel{epi}{\longrightarrow}}
\newcommand{\hto}{\stackrel{hypo}{\longrightarrow}}

\newtheorem{Theorem}{Theorem}
\newtheorem{lemma}{Lemma}[section]
\newtheorem{theorem}[lemma]{Theorem}
\newtheorem{corollary}[lemma]{Corollary}
\newtheorem*{corollary*}{Corollary}

\title{Valuations on Log-Concave Functions}
\author{Fabian Mussnig}
\date{}

\begin{document}
\maketitle

\begin{abstract}
A classification of $\SLn$ and translation covariant Minkowski valuations on log-concave functions is established. The moment vector and the recently introduced level set body of log-concave functions are characterized. Furthermore, analogs of the Euler characteristic and volume are characterized as $\SLn$ and translation invariant valuations on log-concave functions.
\bigskip

{\noindent
2000 AMS subject classification: 26B25 (46B20, 52A21, 52A41, 52B45)}
\end{abstract}

\noindent
A function $\oZ$ defined on a lattice $(\mathcal{L},\mx,\mn)$ and taking values in an abelian semigroup is called a \textit{valuation} if
\begin{equation}
\label{eq:val}
\oZ(f\mx g) + \oZ(f\mn g) = \oZ(f)+\oZ(g),
\end{equation}
for all $f,g\in\mathcal{L}$. A function $\oZ$ defined on a set $\mathcal{S}\subset \mathcal{L}$ is called a valuation if (\ref{eq:val}) holds whenever $f,g,f\mx g,f\mn g \in\mathcal{S}$. In the classical theory, valuations on the set of convex bodies (non-empty, compact, convex sets), $\Kn$, in $\Rn$ are studied, where $\mx$ and $\mn$ denote union and intersection, respectively. Valuations played a critical role in Dehn's solution of Hilbert's Third Problem and have been a central focus in convex geometric analysis. In many cases, well known functions in geometry could be characterized as valuations. For example, a first classification of the Euler characteristic and volume as continuous, $\SLn$ and translation invariant valuations on $\Kn$ was established by Blaschke \cite{blaschke}. The celebrated Hadwiger classification theorem \cite{hadwiger} gives a complete classification of continuous, rotation and translation invariant valuations on $\Kn$ and provides a characterization of intrinsic volumes. Alesker \cite{Alesker01} obtained classification results for translation invariant valuations. Since several important geometric operators like the Steiner point and the moment vector are not translation invariant, also translation covariance played an important role. In particular, Hadwiger~\&~Schneider \cite{hadwiger_schneider} characterized linear combinations of quermassvectors as continuous, rotation and translation covariant vector-valued valuations.\par
In addition to the ongoing research on real-valued valuations on convex bodies \cite{Alesker99,haberl_parapatits_centro,huang_lyz_acta,ludwig_reitzner_annals,li_ma,ludwig_reitzner}, valuations with values in $\Kn$ have attracted interest. Such a map is called a \textit{Minkowski valuation} if the addition in (\ref{eq:val}) is given by Minkowski addition, that is $K+L=\{x+y\,:\,x\in K,\,y\in L\}$ for $K,L\in\Kn$. The first results in this direction were established by Ludwig \cite{ludwig_projection, ludwig_minkowski}. See \cite{haberl, haberl_parapatits_moments,li_leng,wannerer_gln_equi, schuster_wannerer_gln_contra_mink} for some of the pertinent results.\par
More recently, valuations were defined on function spaces. For a space $\mathcal{S}$ of real-valued functions we denote by $f\mx g$ the pointwise maximum of $f$ and $g$ while $f\mn g$ denotes their pointwise minimum. For Sobolev spaces \cite{ludwig_fisher, ludwig_sobolev, ma} and $L^p$ spaces \cite{ludwig_covariance, tsang_val_on_lp, tsang_mink_val_on_lp, ober_minkowski} complete classifications of valuations intertwining with the $\SLn$ were established. For definable functions, an analog to Hadwiger's theorem was proven \cite{baryshnikov_ghrist_wright}. Valuations on convex functions \cite{Alesker_convex,cavallina_colesanti,colesanti_ludwig_mussnig,colesanti_ludwig_mussnig_mink} and quasi-concave functions \cite{colesanti_lombardi, colesanti_lombardi_parapatits} were introduced and classified.\par
The aim of this paper is to establish a classification of $\SLn$ and translation covariant valuations on log-concave functions, that is functions of the form $e^{-u}$, where $u$ is a convex function. Let $\LC$ denote the space of log-concave functions $f:\Rn\to[0,+\infty)$ that are not identically $0$, upper semicontinuous and vanish at infinity. Here a function is said to \textit{vanish at infinity} if
\begin{equation*}
\lim_{|x|\to+\infty} f(x)=0,
\end{equation*} 
where $|x|$ denotes the Euclidean norm of $x$. Furthermore, we equip $\LC$ with the topology associated to hypo-convergence (see Section~\ref{subse:conv_and_log_fct}). It is easy to see that for any $K\in\Kn$ the characteristic function $\Char_K$ is an element of $\LC$. Since
\begin{equation*}
\Char_{K\cup L} = \Char_K \mx \Char_L \qquad \Char_{K\cap L} = \Char_K \mn \Char_L,
\end{equation*}
for all $K,L\in\Kn$ such that also $K\cup L\in\Kn$, valuations on $\LC$ can be seen as a generalization of valuations on $\Kn$. For $f\in\LC$ the level set body $[f]$ is given by
\begin{equation*}
h([f],z)=\int_0^{+\infty} h(\{f \geq t\},z) \d t,
\end{equation*}
for $z\in\Rn$. Here, $h(K,z)=\max\{z\cdot x\,:\, x\in K\}$, where $z \cdot x$ is the standard inner product of $x,z\in\Rn$, denotes the support function of a convex body $K\in\Kn$ and uniquely describes $K$. Moreover, we set $h(\emptyset,z)=0$ for every $z\in\Rn$. The level set body was recently introduced for a more general class of quasi-concave functions in \cite{colesanti_ludwig_mussnig_mink}, where it appeared in the classification of $\SLn$ covariant Minkowski valuations on convex functions. Note, that $[\Char_K]=K$ for every $K\in\Kn$. Furthermore, the moment vector $\vm(f)$ of a log-concave function $f\in\LC$ is defined by
\begin{equation*}
\vm(f)=\int_{\Rn} x\,f(x) \d x,
\end{equation*}
which is an element of $\Rn$ (for details see Section~\ref{se:sln_cov_mink_val_on_lc}). For functions in $L^1(\Rn, |x|\d x)$, the moment vector was introduced and characterized as a Minkowski valuation by Tsang \cite{tsang_mink_val_on_lp}.\par
For $x\in\Rn$, let $\tau_x$ denote the translation $z\mapsto z+x$ on $\Rn$ and let $\mathcal{S}$ be a space of real-valued functions defined on $\Rn$. We call an operator $\oZ:\mathcal{S}\to\Kn$ \textit{translation covariant} if there exists a function $\oZ^0:\mathcal{S}\to\R$ associated with $\oZ$ such that $\oZ(f\circ \tau_x^{-1})=\oZ(f)+\oZ^0(f)x$ for every $f\in\mathcal{S}$ and $x\in\Rn$. Moreover, $\oZ$ is said to be $\SLn$ \textit{covariant} if $\oZ(f\circ\phi^{-1}) = \phi \oZ(f)$ for every $\phi\in\SLn$ and $f\in\mathcal{S}$. Furthermore, a function $\oZ:\LC\to\Kn$ is called \textit{homogeneous of degree} $q\in\R$ if $\oZ(sf)=s^q\oZ(f)$ for every $s>0$ and $f\in\LC$. We say that a map $\oZ:\LC\to\Kn$ is \textit{equi-affinely covariant} if it is translation covariant, $\SLn$ covariant  and homogeneous. For the following result let $n\geq 3$.

\begin{Theorem}
\label{thm:class_cov}
An operator $\oZ:\LC\to\Kn$ is a continuous, equi-affinely covariant Minkowski valuation if and only if there exist constants $c_1,c_2\geq 0$, $c_3\in\R$ and $q>0$ such that
\begin{equation*}
\oZ(f)=c_1[f^q]+c_2(-[f^q])+c_3 \vm(f^q)
\end{equation*}
for every $f\in\LC$.
\end{Theorem}

\noindent
Here, $-K$ denotes the reflection of the body $K\in\Kn$ at the origin, that is $h(-K,z)=h(K,-z)$ for $z\in\Rn$. We will see that Theorem~\ref{thm:class_cov} corresponds to a result for valuations on $\Kn$ (see Corollary~\ref{cor:haberl_covariant}). Denoting by $\oD K = K+(-K)$ the difference body of $K\in\Kn$, we immediately obtain the following corollary, which again corresponds to a result in the theory of valuations on $\Kn$ \cite[Corollary 1.2.]{ludwig_minkowski}.

\begin{corollary*}
An operator $\oZ:\LC\to\Kn$ is a continuous, $\SLn$ covariant, translation invariant and homogeneous Minkowski valuation if and only if there exist constants $c\geq 0$ and $q>0$ such that
\begin{equation*}
\oZ(f) = c \oD [f^q],
\end{equation*}
for every $f\in\LC$.
\end{corollary*}

\noindent
Here, we say that a map $\oZ$ defined on a space $\mathcal{S}$ of real-valued functions on $\Rn$ is \textit{translation invariant} if $\oZ(f\circ\tau_x^{-1}) = \oZ(f)$ for every $f\in\mathcal{S}$ and $x\in\Rn$. We remark that this corollary also follows from \cite[Theorem 2]{colesanti_ludwig_mussnig_mink}, where $\SLn$ covariant Minkowski valuations on convex functions were characterized.\par
In order to prove Theorem~\ref{thm:class_cov} we will need a classification of real-valued valuations on $\LC$ that corresponds to the result by Blaschke mentioned above. A map $\oZ:\mathcal{S}\to\R$ is called $\SLn$ \textit{invariant} if $\oZ(f\circ\phi^{-1})=\oZ(f)$ for every $\phi\in\SLn$ and $f\in\mathcal{S}$. We call an operator $\oZ:\LC\to\R$ \textit{equi-affinely invariant} if $\oZ$ is translation invariant, $\SLn$ invariant and homogeneous, where homogeneity is defined as before. Let $n\geq 2$.

\begin{Theorem}
\label{thm:class_inv}
An operator $\oZ:\LC\to\R$ is a continuous, equi-affinely invariant valuation if and only if there exist constants $c_0,c_n\in\R$ and $q\in\R$, with $q>0$ if $c_n\neq 0$, such that
\begin{equation*}
\oZ(f)=c_0 \big(\max\nolimits_{x\in\Rn} f(x) \big)^q+c_n \int_{\Rn} f^q(x) \d x,
\end{equation*}
for every $f\in\LC$.
\end{Theorem}

\noindent
In Section~\ref{se:sln_inv_vals} we will see that $\max_{x\in\Rn} f(x)$ and $\int_{\Rn} f(x) \d x$ can be seen as functional versions of the Euler characteristic and volume, respectively (see also \cite{bobkov_colesanti_fragala}).

\section{Preliminaries}\label{se:prelim}

We collect some properties of convex bodies and convex functions. Basic references are the books by Schneider \cite{schneider} and  Rockafellar \& Wets \cite{rockafellar_wets}. In addition, we recall definitions and classification results on Minkowski valuations. 

We work in $\R^n$ and denote the canonical basis vectors by $e_1,\dots, e_n$. 
Furthermore, let $\conv(A)$ be the convex hull and $\pos(A)$ the positive hull of $A\subset \R^n$. The space of convex bodies, $\Kn$, is equipped with the \emph{Hausdorff metric}, which is given by
\begin{equation*}
\delta(K,L)=\sup\nolimits_{y\in\Sph} |h(K,y)-h(L,y)|
\end{equation*}
for $K,L\in\Kn$, where $h(K,z)=\max\{z\cdot x: x\in K\}$ is the support function of $K$ at $z\in\Rn$. 
The subspace of convex bodies in $\R^n$ containing the origin is denoted by $\Ko$. Moreover, let $\Pn$ denote the space of convex polytopes in $\R^n$ and $\Po$ the space of convex polytopes containing the origin. All these spaces are equipped with the topology coming from the Hausdorff metric.

For $p\ge 0$, a function $h:\Rn\to \R$ is \emph{p-homogeneous} if $h(t\,z)= t^p\, h(z)$ for $t\ge 0$ and $z\in\Rn$. It is \emph{sublinear} if it is $1$-homogeneous and  $h(x+y)\le h(x) +h(y)$ for $x,y\in\Rn$. Every sublinear function is the support function of a unique convex body.  Note that for the Minkowski sum of  $K,L\in\Kn$, we have
\begin{equation*}
h(K+L,z)=h(K,z)+h(L,z)
\end{equation*}
and in particular $h(K+x,z)=h(K,z)+z \cdot x$ for $x,z\in\Rn$.

\subsection{$\SLn$ Covariant Minkowski Valuations on Convex Bodies}
The \emph{moment body} $\oM K$ of $K\in\Kn$ is defined by
\begin{equation*}
h(\oM K,z) = \int_{K} |x\cdot z| \d x
\end{equation*}
for every $z\in\Sph$. Furthermore, the moment vector $\vm(K)$ of $K$ is given by
\begin{equation*}
h(\vm(K),z) = \int_{K} x\cdot z \d x
\end{equation*}
for every $z\in\Sph$. Note that the moment vector is an element of $\,\Rn$.

\goodbreak
We require the following result where the support function of certain moment bodies and moment vectors is calculated for specific vectors. Let $n\geq 2$.

\begin{lemma}[\!\!\cite{colesanti_ludwig_mussnig_mink}, Lemma 2.3]
\label{le:t_l_h}
For $\lambda>0$ and $T_\lambda=\conv\{0,\lambda\,e_1,e_2, \ldots, e_n\}$, 
\begin{equation*}
\begin{aligned}
&h(T_\lambda,e_1)= \lambda	& &h(-T_\lambda,e_1)=0\\
&h(\vm(T_\lambda),e_1)=\tfrac{\lambda^2}{(n+1)!}	\qquad & &h(\oM T_\lambda,e_1)=\tfrac{\lambda^2}{(n+1)!}.
\end{aligned}
\end{equation*}
\end{lemma}

\noindent
The first classification of $\SLn$ covariant Minkowski valuations was established by Ludwig \cite{ludwig_minkowski}, where also the difference body operator was characterized. The following result is due to Haberl.

\begin{theorem}[\!\!\cite{haberl}, Theorem 6]
\label{thm:haberl_covariant}
For $n\geq 3$, a map $\oZ:\Ko\to\Kn$ is a continuous, $\SLn$ covariant Minkowski valuation if and only if there exist constants $c_1,c_2,c_4\geq 0$ and $c_3\in\R$ such that
\begin{equation*}
\oZ K=c_1 K + c_2 (-K) + c_3 \vm(K) + c_4 \oM(K),
\end{equation*}
for every $K\in\Ko$.
\end{theorem}

\noindent
We say that a Minkowski valuation $\oZ:\Kn\to\Kn$ is \emph{translation covariant} if there exists a function $\oZ^0:\Kn\to\R$ associated with $\oZ$ such that
\begin{equation*}
\oZ(K+x)=\oZ(K)+\oZ^0(K)x,
\end{equation*}
for every $K\in\Kn$ and $x\in\Rn$. Since several important geometric operators have this property, translation covariant valuations have attracted interest. For example, the identity on $\Kn$ and the reflection $K\mapsto -K$ are translation covariant. Furthermore, for $z\in\Sph$ we have
\begin{equation}
\label{eq:mv_trans}
h(\vm(K+x),z)=\int_{K+x} y \cdot z  \d y = \int_K (y+x)\cdot z \d y = \int_K y\cdot z \d y + V_n(K)\, (z\cdot x),
\end{equation}
and hence $\vm(K+x)=\vm(K)+V_n(K)x$ for every $K\in\Kn$ and $x\in\Rn$. Based on Schneider's characterization of the Steiner point \cite{schneider_steiner_point}, Hadwiger~\&~Schneider \cite{hadwiger_schneider} proved that the quermassvectors form a basis of the space of continuous, rotation and translation covariant vector-valued valuations. In \cite{mcmullen_weakly_cont_val}, McMullen characterized weakly continuous and translation covariant vector-valued valuations on convex polytopes, extending a previous result by Hadwiger \cite{hadwiger1952}. In his result the intrinsic moment vectors of the faces of a polytope appear. For further results on translation covariant valuations see \cite{mcmullen_schneider, mcmullen_handbook}.

\begin{corollary}
\label{cor:haberl_covariant}
For $n\geq 3$, a map $\oZ:\Kn\to\Kn$ is a continuous, $\SLn$ and translation covariant Minkowski valuation if and only if there exist constants $c_1,c_2\geq 0$ and $c_3\in\R$ such that
\begin{equation}
\label{eq:equi_aff_val}
\oZ K = c_1 K + c_2 (-K) + c_3 \vm(K),
\end{equation}
for every $K\in\Kn$.
\end{corollary}
\begin{proof}
We have already seen in Theorem~\ref{thm:haberl_covariant} and (\ref{eq:mv_trans}) that (\ref{eq:equi_aff_val}) defines a continuous, $\SLn$ and translation covariant Minkowski valuation. Conversely, let $\oZ$ be a continuous Minkowski valuation on $\Kn$ that is $\SLn$ and translation covariant. Obviously, the restriction of $\oZ$ to $\Ko$ is a continuous, $\SLn$ covariant Minkowski valuation. Hence, by Theorem \ref{thm:haberl_covariant} there exist constants $c_1,c_2,c_4\geq 0$ and $c_3\in\R$ such that
\begin{equation}
\label{eq:z_on_ko}
\oZ K = c_1 K + c_2 (-K) + c_3 \vm(K) + c_4 \oM K,
\end{equation}
for every $K\in\Ko$. Define the polytope $P$ as $
P:=[-e_1,2 e_1]+[0,e_2] + \cdots [0,e_n]$ and note that $P, P+e_1, P-e_1 \in\Ko$. By the translation covariance of $\oZ$ we obtain
\begin{align*}
\begin{split}
\oZ(P)+\oZ^0(P)e_1 &= \oZ(P+e_1)\\
&= c_1 P + c_1 e_1 + c_2 (-P) - c_2 e_1 + c_3 \vm(P) + V_n(P) e_1 +c_4 \oM (P+e_1),\\
\oZ(P)-\oZ^0(P)e_1 &= \oZ(P-e_1)\\
&= c_1 P - c_1 e_1 + c_2 (-P) + c_2 e_1 + c_3 \vm(P) - V_n(P) e_1 + c_4 \oM (P-e_1).
\end{split}
\end{align*}
Adding these equations shows that
\begin{align*}
2 \oZ(P) &= \oZ(P+e_1)+\oZ(P-e_1)\\
&= 2 c_1 P + 2 c_2 (-P) + 2 c_3 \vm(P) + c_4 ( \oM (P+e_1) + \oM (P-e_1)).
\intertext{On the other hand by (\ref{eq:z_on_ko})}
2 \oZ(P) &= 2 c_1 P + 2c_2 (-P) + 2 c_3 \vm(P) + 2 c_4 \oM P. 
\end{align*}
Evaluating and comparing the support functions of these two representations of $2\oZ(P)$ at $e_1$ gives
\begin{equation*}
2 c_4 \tfrac 52 = c_4 ( \tfrac 92 + \tfrac 52),
\end{equation*}
and therefore $c_4=0$. Furthermore, this shows that $\oZ^0(K)=c_1-c_2 + c_3 V_n(K)$ for every $\K\in\Ko$. Now, fix an arbitrary $K\in\Kn$. Then, there exist $K_o\in\Ko$ and $x\in\Rn$ such that $K=K_o+x$. By the properties of $\oZ$ this gives
\begin{align*}
\begin{split}
\oZ(K)&=\oZ(K_o+x)\\&=\oZ(K_o)+\oZ^0(K_o)x\\
&=c_1 K_o + c_2 (-K_o) + c_3 \vm(K_o) + (c_1-c_2+V_n(K_o))x\\
&= c_1 K + c_2 (-K) + c_3 \vm(K).
\end{split}
\end{align*}
\end{proof}

\subsection{Convex and Log-Concave Functions}
\label{subse:conv_and_log_fct}
Let $f=e^{-u}\in\LC$. Then, $u:\Rn\to(-\infty,+\infty]$ is convex, lower-semicontinuous, proper and coercive. Here we say that $u$ is \emph{proper} if $u\neq +\infty$. Furthermore, $u$ is called \emph{coercive} if
\begin{equation*}
\lim_{|x|\to+\infty} u(x)=+\infty.
\end{equation*}
The set of all such functions $u$ will be denoted by $\CV$. Obviously,
\begin{equation*}
\LC=\{e^{-u}\,:\,u\in\CV\}.
\end{equation*}
Furthermore, for a function $u\in\CV$, the \emph{domain}, $\dom u=\{x\in\Rn\,:\,u(x)<+\infty\}$, of $u$ is a convex subset of $\Rn$. Moreover, its \emph{epigraph}
\begin{equation*}
\epi u =\{(x,t)\in\Rn\times \R\,:\, u(x)\leq t\},
\end{equation*}
is a closed convex subset of $\Rn\times\R$. Hence, every function $u\in\CV$ attains its minimum and $\min_{x\in\Rn} u(x)>-\infty$. Furthermore, for $t\in\R$, the \emph{sublevel set}
\begin{equation*}
\{u\leq t\}=\{x\in\Rn\,:\,u(x)\leq t\},
\end{equation*}
is either a convex body or empty. Note that for $u,v\in\CV$ and $t\in\R$
\begin{equation*}
\{u\wedge v \leq t\} = \{u\leq t\} \cup \{v\leq t\}\qquad\text{ and }\qquad \{u\vee v\leq t\}= \{u\leq t\} \cap \{v\leq t\},
\end{equation*}
where for $u\wedge v\in\CV$ all occurring sublevel sets are either empty or in $\Kn$. Equivalently, every function $f\in\LC$ attains its maximum and for $0<t\leq \max_{x\in\Rn} f(x)$ the \emph{superlevel set}
\begin{equation*}
\{f\geq t\}=\{x\in\Rn\,:\,f(x)\geq t\},
\end{equation*}
is a convex body. Moreover, for every $f,g\in\LC$ and $t>0$
\begin{equation*}
\{f\wedge g \geq t\} = \{f\geq t\} \cap \{g\geq t\}\qquad\text{ and }\qquad \{f\vee g\geq t\}= \{f\geq t\} \cap \{g\geq t\}.
\end{equation*}
\noindent
We equip $\CV$ with the topology associated to epi-convergence, which is the standard topology for a space of extended real-valued convex functions. Here a sequence $u_k: \Rn\to (-\infty, \infty]$ is \emph{epi-convergent} to $u:\Rn\to (-\infty, \infty]$ if for all $x\in\Rn$ the following conditions hold:
\begin{itemize}
	\item[(i)] For every sequence $x_k$ that converges to $x$,
			\begin{equation*}\label{eq:gc_inf}
				u(x) \leq \liminf\nolimits_{k\to \infty} u_k(x_k).
			\end{equation*}
	\item[(ii)] There exists a sequence $x_k$ that converges to $x$ such that
			\begin{equation*}\label{eq:gc_sup}
				u(x) = \lim\nolimits_{k\to\infty} u_k(x_k).
			\end{equation*}
\end{itemize}
\vskip -8pt
In this case we write $u=\elim_{k\to\infty} u_k$ and $u_k \eto u$. We remark that epi-convergence is also called $\Gamma$-convergence. Correspondingly, we say that a sequence $f_k$ in $\LC$ is \emph{hypo-convergent} to $f\in\LC$ if there exist $u_k,u\in\CV$ such that $f_k=e^{-u_k}$ for every $k\in\N$, $f=e^{-u}$ and $u_k\eto u$. In this case we write $f=\hlim_{k\to\infty} f_k$ and $f_k \hto f$.

The following results connect epi-convergence and Hausdorff convergence of sublevel sets. We say that $\{u_k \leq t\} \to \emptyset$ as $k\to\infty$ if there exists $k_0\in\N$ such that $\{u_k \leq t\} = \emptyset$ for all $k\geq k_0$.

\begin{lemma}[\!\cite{colesanti_ludwig_mussnig}, Lemma 5]
\label{le:lk_conv}
Let $u_k,u\in\CV$. If $u_k\eto u_k$, then $\{u_k\leq t\} {\to} \{u\leq t\}$ for every $t\in\R$ with $t\neq \min_{x\in\Rn} u(x)$.
\end{lemma}

\begin{lemma}[\!\!\cite{rockafellar_wets}, Proposition 7.2]
\label{le:epi_lvl_sets} 
Let $u_k, u \in\CV$. If for each $t\in\R$ there exists a sequence $t_k$ of reals convergent to $t$ with $\{u_k\leq t_k\} \to \{u\leq t\}$, then $u_k \eto u$.  \end{lemma}
\noindent
Furthermore, the so-called cone property and uniform cone property will be useful in order to show that certain integrals converge.

\begin{lemma}[\!\!\cite{colesanti_fragala}, Lemma 2.5]
\label{le:cone}
For $u\in\CV$ there exist constants $a,b \in \R$ with $a >0$ such that
\begin{equation*}
u(x)>a|x|+b
\end{equation*}
for every $x\in\Rn$.
\end{lemma}

\begin{lemma}[\!\!\cite{colesanti_ludwig_mussnig}, Lemma 8]
\label{le:uniform_cone}
Let $u_k,u\in\CV$. If $u_k \eto u$, then there exist constants $a,b\in\R$ with $a>0$ such that
\begin{equation*}
u_k(x)>a|x|+b\quad \text{and}\quad u(x)> a|x|+b,
\end{equation*}
for every $k\in\N$ and $x\in\Rn$.
\end{lemma}
\noindent
Next, we introduce some special elements of $\CV$. For $K\in\Ko$, we define the convex function $\l_K:\Rn\to[0,+\infty]$ by
\begin{equation*}
\epi \l_K=\pos (K\times \{1\}).
\end{equation*}
This means that the epigraph of $\l_K$ is a cone with apex at the origin and $\{\l_K\leq t \}=t \, K$ for all $t \geq 0$. It is easy to see that $\l_K$ is an element of $\CV$ for  $K\in\Ko$. Also the (convex) indicator function $\Ind_K$ for $K\in\K^n$ belongs to $\CV$, where
\begin{equation*}
\Ind_K(x)=\begin{cases} 0,\quad &\text{if } x\in K\\
+\infty, &\text{if } x\notin K. \end{cases}
\end{equation*}
Observe, that $e^{-\Ind_K} = \Char_K$ for every $K\in\Kn$.\par
Let $f,g\in\LC$ be such that $f\mx g\in\LC$ and let $\oZ:\LC\to\langle A,+\rangle$, where $\langle A,+\rangle$ is an abelian semigroup. By definition, there exist functions $u,v\in\CV$ such that $u\mn v\in\CV$ and $f=e^{-u}$ and $g=e^{-v}$. Since
\begin{equation*}
f\mx g = e^{-(u\mn g)}\qquad \text{and}\qquad f\mn g = e^{-(u\mx g)},
\end{equation*}
the map $\oZ$ is a valuation if and only if $\oY:\CV\to\langle A,+\rangle$ is a valuation, where
\begin{equation*}
\oY(u)=\oZ(e^{-u}),
\end{equation*}
for every $u\in\CV$. Hence, studying valuations on $\LC$ is equivalent to studying valuations on $\CV$ and it will be convenient for us to switch between these points of view. By the definition of hypo-convergence on $\LC$, the valuation $\oZ$ is continuous if and only if $\oY$ is continuous. Furthermore, for $x\in\Rn$ we have $f\circ\tau_x^{-1}=e^{-u\circ\tau_x^{-1}}$. Hence, $\oZ$ is translation invariant if and only if $\oY$ is translation invariant. Similarly, translation covariance, $\SLn$ invariance and $\SLn$ covariance are equivalent for valuations on $\LC$ and their counterparts on $\CV$.\par
The next result, which is based on \cite{ludwig_sobolev}, shows that a valuation on $\CV$ is uniquely determined by its behaviour on certain functions.

\begin{lemma}[\!\!\cite{colesanti_ludwig_mussnig}, Lemma 17]
\label{le:reduction}
Let $\langle A,+\rangle$ be a topological abelian semigroup with cancellation law and let $\oY_1, \oY_2:\CV\to\langle A,+\rangle$ be continuous valuations. If $\oY_1(\l\circ\tau_x^{-1})=\oY_2(\l\circ\tau_x^{-1})$ for every $\l\in\{\l_P+t:\, P\in\Po,\, t\in\R\}$ and every $x\in\Rn$, then $\oY_1 \equiv \oY_2$ on $\CV$.
\end{lemma}
\noindent
We remark, that in \cite[Lemma 17]{colesanti_ludwig_mussnig} it is assumed that the valuations are translation invariant. However, translation invariance itself is not needed for the proof and it is easy to see that this more general statement holds.

\section{$\SLn$ Invariant Real-Valued Valuations on $\LC$}
\label{se:sln_inv_vals}
We denote by $V_0$ the Euler characteristic, that is $V_0(K)=1$ for every $K\in\Kn$ and $V_0(\emptyset)=0$. Since for $f\in\LC$ the level sets $\{f\geq t\}$ are convex bodies for every $0<t\leq \max_{x\in\Rn} f$, it makes sense to consider
\begin{equation*}
V_0(f):=\int_{0}^{+\infty} V_0(\{f\geq t\}) \d t = \max\nolimits_{x\in\Rn} f(x).
\end{equation*}
Furthermore, denoting by $V_n$ the $n$-dimensional volume or Lebesgue measure, and assuming that the integrals converge, we have by the layer-cake principle
\begin{equation*}
V_n(f):=\int_{0}^{+\infty} V_n(\{f\geq t\}) \d t = \int_{\Rn} f(x) \d x.
\end{equation*}
We remark that this notion for the volume of a (log-concave) function is frequently used and there are several examples of functional counterparts of geometric inequalities, in which the volume $V_n(K)$ of a convex body $K$ is replaced by the integral $\int f$ of a function $f$. For example, the Pr\'ekopa-Leindler inequality is the functional analog of the Brunn-Minkowski inequality \cite{prekopa, leindler}. Furthermore, functional versions of quermassintegrals where recently introduced for quasi-concave functions \cite{bobkov_colesanti_fragala,milman_rotem}.\par
We need the following result where the volume operator of a specific function is calculated. Let $n\geq 2$.

\begin{lemma}
\label{le:vn_l_t_lambda}
For $\lambda\geq 0$, $q>0$ and $T_\lambda=\conv\{0,\lambda \,e_1,e_2, \ldots, e_n\}$, 
\begin{equation*}
V_n(e^{-q \l_{T_\lambda}})= \tfrac{\lambda}{q^n}.
\end{equation*}
\end{lemma}
\begin{proof}
By definition we have
\begin{equation*}
V_n(e^{-q \l_{T_\lambda}}) = \int_0^1 V_n(\{e^{-q \l_{T_\lambda}} \leq t\}) \d t = \int_0^1 V_n(\{\l_{T_\lambda} \leq -\tfrac{\log t}{q}\}) \d t.
\end{equation*}
Using the substitution $s=-\tfrac{\log t}{q}$ we have $\d t = -q e^{-qs} \d s$ and therefore
\begin{equation*}
V_n(e^{-q \l_{T_\lambda}}) = q \int_0^{+\infty} V_n(\{\l_{T_\lambda}\leq s\})\, e^{-q s} \d s.
\end{equation*}
By definition, $\{\l_{T_\lambda}\leq s\}= s\,T_\lambda$ for every $s\geq 0$. Hence,
\begin{equation*}
V_n(\{\l_{T_\lambda}\leq s\}) = s^n V_n(T_\lambda) = s^n \tfrac{\lambda}{n!}.
\end{equation*}
This gives
\begin{equation*}
V_n(e^{-q \l_{T_\lambda}}) = \tfrac{\lambda}{n!} \int_0^{+\infty} s^n e^{-qs} q \d s = \tfrac{1}{q^n}\tfrac{\lambda}{n!} \int_0^{+\infty} (qs)^n e^{-qs} q \d s = \tfrac{1}{q^n}\tfrac{\lambda}{n!} \int_0^{+\infty} r^n e^{-r} \d r = \tfrac{\lambda}{q^n}.
\end{equation*}
\end{proof}
\noindent
The following results are mostly deduced from \cite{colesanti_ludwig_mussnig}, where analogs of $V_0$ and $V_n$ on $\CV$ were studied. 

\begin{lemma}
\label{le:v0_on_lc_is_a_val}
For every $q\in\R$, the map
\begin{equation*}
f\mapsto V_0(f)^q 
\end{equation*}
is a continuous, equi-affinely invariant valuation on $\LC$ that is homogeneous of degree $q$.
\end{lemma}
\begin{proof}
\begin{sloppypar}
Since
\begin{equation*}
\left(\max_{x\in\Rn} s\, f(x)\right)^q = s^q \left(\max_{x\in\Rn} f(x)\right)^q,
\end{equation*}
for every $f\in\LC$ and $s>0$, the map $f\mapsto V_0(f)^q $ is homogeneous of degree $q$. By \break \cite[Lemma 12]{colesanti_ludwig_mussnig} it is a continuous, $\SLn$ and translation invariant valuation.
\end{sloppypar}
\end{proof}

\begin{lemma}
\label{le:vn_on_lc_is_a_val}
For every $q>0$, the map
\begin{equation*}
\label{eq:vn_f_q}
f\mapsto V_n(f^q)
\end{equation*}
is a continuous, equi-affinely invariant valuation on $\LC$ that is homogeneous of degree $q$.
\end{lemma}
\begin{proof}
By \cite[Lemmas 15 \& 16]{colesanti_ludwig_mussnig}, the map $f\mapsto V_n(f^q)$ is a well-defined continuous, $\SLn$ and translation invariant valuation on $\LC$. Since
\begin{equation*}
\int_{\Rn} (sf)^q(x) \d x = s^q \int_{\Rn} f^q (x) \d x,
\end{equation*}
for every $f\in\LC$ and $s>0$, it is homogeneous of degree $q$.
\end{proof}

\section{Classification of $\SLn$ Invariant Real-Valued Valuations}

In the following Lemma we collect some results that were proved in \cite{colesanti_ludwig_mussnig}. Let $n\geq 2$.
\begin{lemma}
\label{le:results_on_sln_inv_val}
If $\oY:\CV\to\R$ is a continuous, $\SLn$ and translation invariant valuation, then there exist continuous functions $\zeta_0,\zeta_n,\psi_n:\R\to\R$ such that
\begin{eqnarray*}
\oY(\l_K+t)&=&\zeta_0(t)+\psi_n(t) V_n(K),\\
\oY(\Ind_K+t)&=&\zeta_0(t)+\zeta_n(t) V_n(K),
\end{eqnarray*}
for every $K\in\Ko$ and $t\in\R$. Furthermore, $\lim_{t\to+\infty} \psi_n(t)=0$ and \begin{equation*}
\zeta_n(t)=\frac{(-1)^n}{n!}\frac{\d^n}{\d t^n}\psi_n(t),
\end{equation*}
for every $t\in\R$. Moreover, $\oY$ is uniquely determined by $\zeta_0$ and $\zeta_n$.
\end{lemma}

\subsection{Proof of Theorem~\ref{thm:class_inv}}
By Lemmas \ref{le:v0_on_lc_is_a_val} and \ref{le:vn_on_lc_is_a_val}, the operator
\begin{equation*}
f\mapsto c_0 V_0(f)^q + c_n V_n(f^q),
\end{equation*}
defines a continuous, equi-affinely invariant valuation on $\LC$ for every $c_0,c_n\in\R$ and $q\in\R$, when $q>0$ if $c_n\neq 0$.\par
Conversely, let $\oZ:\LC\to\R$ be a continuous, equi-affinely invariant valuation and let $\oY$ be the corresponding valuation on $\CV$, that is $\oY(u)=\oZ(e^{-u})$ for every $u\in\CV$. Then $\oY$ is continuous, $\SLn$ and translation invariant. Furthermore,
\begin{equation*}
\oY(u+t)=\oZ(e^{-u-t}) = (e^{-t})^q \oZ(e^{-u}) = e^{-qt} \oY(u),
\end{equation*}
for every $u\in\CV$ and $t\in\R$, where $q\in\R$ denotes the degree of homogeneity of $\oZ$. Let $\zeta_0,\zeta_n,\psi_n$ be the functions from Lemma \ref{le:results_on_sln_inv_val}. We have,
\begin{equation*}
\zeta_0(t)=\oY(\Ind_{\{0\}}+t)=e^{-qt} \oY(\Ind_{\{0\}}),
\end{equation*}
for every $t\in\R$. Hence, there exists a constant $c_0\in\R$ such that $\zeta_0(t)=c_0 e^{-qt}$ for every $t\in\R$. Furthermore, let $K\in\Ko$ such that $V_n(K)>0$. Then,
\begin{equation*}
e^{-qt}\oY(\l_K)=\oY(\l_K+t)=\zeta_0(t)+\psi_n(t)V_n(K)=c_0 e^{-qt} + \psi_n(t)V_n(K),
\end{equation*}
for every $t\in\R$. Hence, there exists a constant $\widetilde{c}_n\in\R$ such that $\psi_n(t)=\widetilde{c}_n\,e^{-qt}$ for every $t\in\R$. Since $\lim_{t\to+\infty} \psi_n(t)=0$, we must have $q>0$ or $\widetilde{c}_n=0$. Moreover,
\begin{equation*}
\zeta_n(t)=\frac{(-1)^n}{n!}\frac{\d^n}{\d t^n} \psi_n(t) = \frac{\widetilde{c}_n\, q^n}{n!} e^{-qt}=:c_n\,e^{-qt},
\end{equation*}
for every $t\in\R$. For $t\in\R$, let $s=e^{-t}$. We have
\begin{align*}
\oZ(s \, \Char_K) = \oY(\Ind_K+t) &= c_0\, e^{-qt}+c_n\, e^{-qt} V_n(K)\\
&= c_0\, s^q + c_n\, s^q V_n(K)\\
&=c_0 \left(\int_0^{+\infty} V_0(\{s\,\Char_K \geq r\}) \d r\right)^q + c_n \int_0^{+\infty} V_n(\{(s\,\Char_K)^q \geq r\}) \d r\\
&=c_0 V_0(s\, \Char_K)^q+c_n V_n((s\, \Char_K)^q),  
\end{align*}
for every $K\in\Kn$. Since $\oY$ is uniquely determined by its values on indicator functions and
\begin{equation*}
f\mapsto c_0 V_0(f)^q + c_n V_n(f^q)
\end{equation*}
defines a continuous, equi-affinely invariant valuation, the proof is complete.\hfill\qedsymbol

\section{$\SLn$ Covariant Minkowski Valuations on $\LC$}
\label{se:sln_cov_mink_val_on_lc}
In this section we discuss the operators that appear in Theorem~\ref{thm:class_cov} and show that they are continuous, equi-affinely covariant Minkowski valuations.\par
In \cite{bobkov_colesanti_fragala} it is proposed to generalize a function $\Phi:\Kn\to[0,+\infty)$ to $\LC$ via
\begin{equation*}
\Phi(f)=\int_0^{+\infty} \Phi(\{f\geq t\}) \d t,
\end{equation*}
for $f\in\LC$. Note, that this construction implicitly uses the general convention $\Phi(\emptyset)=0$. Following this approach, the level set body $[f]$ of $f\in\LC$ is the convex body that is defined via
\begin{equation*}
h([f],z)=\int_{0}^{+\infty} h(\{f\geq t\},z) \d t,
\end{equation*}
for every $z\in\Rn$.

\begin{lemma}
\label{le:id_on_lc_is_a_val}
For every $q>0$, the map
\begin{equation}
\label{eq:lvl_set_body_f_q}
f\mapsto [f^q]
\end{equation}
is a continuous, equi-affinely covariant Minkowski valuation on $\LC$ that is homogeneous of degree $q$.
\end{lemma}
\begin{proof}
By \cite[Lemma 7.2]{colesanti_ludwig_mussnig_mink}, the map $f\mapsto [f^q]$ is a well-defined, continuous, $\SLn$ covariant Minkowski valuation on $\LC$. Furthermore, for $s>0$, $x,z\in\Rn$ and $f\in\LC$, we have
\begin{align*}
h([(s\, f\circ\tau_x^{-1})^q],z) &= \int_{0}^{+\infty} h(\{(s\,f\circ\tau_x^{-1})^q \geq t\},z)\d t\\
&= s^q \int_0^{+\infty} h(\tau_x\{f^q\geq t\},z)\d t\\
&= s^q \int_0^{+\infty} h(\{f^q\geq t\},z)\d t + s^q\big(\max\nolimits_{x\in\Rn} f^q(x)\big)\, (x\cdot z)\\
&= s^q \, h([f^q],z) + s^q \, V_0(f^q)\,(x\cdot z).
\end{align*}
Hence, (\ref{eq:lvl_set_body_f_q}) is homogeneous of degree $q$ and translation covariant.
\end{proof}
\noindent
The next lemma will allow us to give a definition of the moment vector for functions in $\LC$.
\begin{lemma}
\label{le:mv_finite}
For every $f\in\LC$ and $z\in\Sph$,
\begin{equation*}
\int_0^{+\infty} | h(\vm(\{f \geq t\}),z) | \d t  < +\infty.
\end{equation*}
\end{lemma}
\begin{proof}
Observe, that for $K\in\Kn$ and $z\in\Sph$
\begin{equation*}
|h(\vm(K),z)| = \left|\int_{K} x\cdot z \d x\right| \leq V_n(K)\,\max\nolimits_{y\in\Sph} |h(K,y)|.
\end{equation*}
Fix $f\in\LC$ and let $u\in\CV$ be such that $f=e^{-u}$. By Lemma \ref{le:cone}, there exist constants $a,b\in\R$ with $a>0$ such that
\begin{equation*}
u(x)> v(x) = a |x|+ b,
\end{equation*}
for every $x\in\Rn$. Hence, for $g=e^{-v}\in\LC$ we have $f< g$ pointwise and therefore
\begin{equation*}
\{f\geq t\}\subset \{g \geq t\} = \big\{x\,:\, |x|\leq \tfrac{-\log t-b}{a}\big\}
\end{equation*}
for every $0<t\leq e^{-b}$. This gives
\begin{equation*}
|h(\vm(\{f\geq t\}),z)| \leq V_n(\{g \geq t\}) \, \max\nolimits_{y\in\Sph} |h(\{g \geq t\},y)| = \tfrac{v_n}{a^{n+1}} \big(-\log t-b \big)^{n+1},
\end{equation*}
for every $0<t\leq e^{-b}$ and $z\in\Sph$, where $v_n$ denotes the volume of the $n$-dimensional unit ball. Thus, using the substitution $t=e^{-s}$, we obtain
\begin{align*}
\begin{split}
\int_0^{+\infty} |h(\vm(\{f\geq t\}),z)| \d t & \leq \tfrac{v_n}{a^{n+1}} \int_0^{e^{-b}} (-\log t-b)^{n+1} \d t\\
& \leq \tfrac{v_n}{a^{n+1}} \int_b^0 (s-b)^{n+1} e^{-s} \d s < +\infty,
\end{split}
\end{align*}
for every $z\in\Rn$.
\end{proof}
\noindent
By Lemma~\ref{le:mv_finite}, the integral
\begin{equation}
\label{eq:int_moment_vectors}
\int_0^{+\infty} h(\vm(\{f\geq t\}),z)\d t
\end{equation}
converges for every $f\in\LC$ and $z\in\Rn$. Since each of the support functions
\begin{equation*}
z\mapsto h(\vm(\{f\geq t\}),z)
\end{equation*}
is sublinear, it is easy to see that (\ref{eq:int_moment_vectors}) defines a sublinear function in $z$ and thus is the support function of a convex body $\vm(f)\in\Kn$. Using the definition of the moment vector and the layer-cake principle, we obtain
\begin{equation*}
h(\vm(f),z) = \int_0^{+\infty} \int_{\{f\geq t\}} x \cdot z \d x \d t = \int_0^{+\infty} \int_{\Rn} \Char_{\{ f\geq t\}}(x)\,(x\cdot z) \d x \d t = \int_{\Rn} f(x)\, (x\cdot z) \d x.
\end{equation*} 
Hence,
\begin{equation*}
\vm(f)=\int_{\Rn} f(x)\,x \d x
\end{equation*}
is an element of $\Rn$ and will be called the moment vector of $f\in\LC$.

\begin{lemma}
\label{le:mv_on_lc_is_a_val}
For every $q>0$, the map
\begin{equation}
\label{eq:mv_f_q}
f \mapsto \vm(f^q)
\end{equation}
is a continuous, equi-affinely covariant Minkowski valuation on $\LC$ that is homogeneous of degree $q$.
\end{lemma}
\begin{proof}
Since $f^q\in\LC$ for every $f\in\LC$ and $q>0$, the map $f\mapsto \vm(f^q)$ is well-defined. For $\phi\in\SLn$ we have
\begin{equation*}
\vm((f\circ \phi^{-1})^q) = \int_{\Rn} (f^q\circ \phi^{-1})(x)\, x \d x = \int_{\Rn} f^q(x) \, \phi x \d x = \phi \vm(f^q),
\end{equation*}
which shows $\SLn$ covariance. Furthermore, for $x\in\Rn$ we obtain
\begin{equation*}
\vm((f\circ\tau_x^{-1})^q) = \int_{\Rn} f^q(y-x)\,y \d y = \int_{\Rn} f^q(y)\,y \d y + x \int_{\Rn} f^q(y) \d y = \vm(f^q) + V_n(f^q)\, x,
\end{equation*}
and for $s>0$
\begin{equation*}
\vm((sf)^q) = \int_{\Rn} (sf)^q(x)\, x \d x = s^q \int_{\Rn} f^q(x) \, x \d x = s^q \vm(f^q).
\end{equation*}
Hence, (\ref{eq:mv_f_q}) is equi-affinely covariant. In order to show the valuation property, let $f,g\in\LC$ such that $f\mx g\in\LC$. Then,
\begin{eqnarray*}
\vm((f\mn g)^q) &=& \int_{\{f \leq g\}} f^q(x)\, x \d x + \int_{\{f > g\}} g^q(x)\, x \d x\\
\vm((f\mx g)^q) &=& \int_{\{f \leq g\}} g^q(x)\, x \d x + \int_{\{f > g\}} f^q(x)\, x \d x.
\end{eqnarray*}
Hence,
\begin{equation*}
\vm((f\mn g)^q) + \vm((f\mx g)^q) = \vm(f^q) +\vm(g^q).
\end{equation*}
It remains to show continuity. For $f_k,f\in\LC$ such that $f_k\hto f$, there exist $u_k,u\in\LC$ such that $f_k=e^{-u_k}$ for every $k\in\N$, $f=e^{-u}$ and $u_k\eto u$. By Lemma~\ref{le:uniform_cone}, there exist $a>0$ and $b\in\R$ such that
\begin{equation*}
u_k(x)>a|x|+b\qquad\text{and}\qquad u(x)>a|x|+b,
\end{equation*}
for every $k\in\N$ and $x\in\Rn$. Similar as in the proof of Lemma~\ref{le:mv_finite}, this gives
\begin{eqnarray*}
|h(\vm(\{f\geq t\}),\cdot)| &\leq& \tfrac{v_n}{a^{n+1}}(-\log t - b)^{n+1}\\
|h(\vm(\{f_k\geq t\}),\cdot)| &\leq& \tfrac{v_n}{a^{n+1}}(-\log t - b)^{n+1},
\end{eqnarray*}
which shows that these functions are dominated by an integrable function. Furthermore, Lemma~\ref{le:epi_lvl_sets} and the continuity of the moment vector on $\Kn$ imply that\break$h(\vm(\{f_k\geq t\}),\cdot)\to h(\vm(\{f\geq t\}),\cdot)$ pointwise for every $t\neq \max_{x\in\Rn} f(x)$. Hence, by the dominated convergence theorem, we have
\begin{equation*}
h(\vm(f_k),\cdot) = \int_{0}^{+\infty} h(\vm(\{f_k\geq t\},\cdot) \d t \longrightarrow \int_{0}^{+\infty} h(\vm(\{f\geq t\},\cdot) \d t = h(\vm(f),\cdot),
\end{equation*}
pointwise, which implies Hausdorff convergence of $\vm(f_k)$ to $\vm(f)$. The claim now follows, since $f\mapsto f^q$ is continuous and $f_k^q \hto f^q$.
\end{proof}

\section{Classification of $\SLn$ Covariant Minkowski Valuations}
\begin{sloppypar}
The next result extends the basic observation that the associated function $\oZ^0:\Kn\to\Rn$ of a translation covariant Minkowski valuation $\oZ:\Kn\to\Kn$ is a translation invariant real-valued valuation. See for example \cite[Lemma 10.5]{mcmullen_schneider} for a corresponding result on vector-valued valuations. Similarly, $\SLn$ covariance of $\oZ$ implies $\SLn$ invariance of $\oZ^0$. Hence, it is no coincidence that the associated function of the Minkowski valuation described in Corollary~\ref{cor:haberl_covariant} is a linear combination of the Euler characteristic and volume.
\end{sloppypar}

\begin{lemma}
\label{le:z0_is_a_val}
If $\oZ:\LC\to\Kn$ is a continuous, equi-affinely covariant Minkowski valuation, then its associated function $\oZ^0:\LC\to\R$ is a continuous, equi-affinely invariant valuation. Furthermore, $\oZ$ and $\oZ^0$ have the same degree of homogeneity.
\end{lemma}
\begin{proof}
Let $x\in\Rn \backslash\{0\}$ and $f,g\in\LC$ be such that $f\vee g\in\LC$. Since
\begin{equation*}
(f\circ \tau_x^{-1}) \vee (g\circ \tau_x^{-1}) = (f \vee g)\circ \tau_x^{-1},\quad (f\circ \tau_x^{-1}) \wedge (g\circ \tau_x^{-1}) = (f \wedge g)\circ \tau_x^{-1},
\end{equation*}
it follows from the translation covariance and the valuation property of $\oZ$ that
\begin{align*}
\oZ(f\circ \tau_x^{-1})+\oZ(g\circ \tau_x^{-1})&=\oZ((f\vee g)\circ\tau_x^{-1})+\oZ((f\wedge g)\circ \tau_x^{-1})\\
&= \oZ(f\vee g) + \oZ(f\wedge g)+\oZ^0(f\vee g)x+\oZ^0(f\wedge g)x.
\end{align*}
On the other hand,
\begin{align*}
\oZ(f\circ \tau_x^{-1})+\oZ(g\circ \tau_x^{-1}) &= \oZ(f)+\oZ^0(f)x+\oZ(g)+\oZ^0(g)x\\
&= \oZ(f\vee g) + \oZ(f\wedge g) + \oZ^0(f)x+\oZ^0(g)x.
\end{align*}
Hence, $\oZ^0$ is a valuation. Now, for arbitrary $y\in\Rn$, we have
\begin{align*}
\begin{split}
\oZ(f)+\oZ^0(f)x+\oZ^0(f)y&=\oZ(f\circ\tau_{x+y}^{-1})\\
&=\oZ(f\circ\tau_y^{-1}\circ\tau_x^{-1})\\
&=\oZ(f\circ \tau_y^{-1})+\oZ^0(f\circ \tau_y^{-1})x\\
&=\oZ(f)+\oZ^0(f)y+\oZ^0(f\circ \tau_y^{-1})x,
\end{split}
\end{align*}
and therefore $\oZ^0(f)=\oZ^0(f\circ\tau_y^{-1})$. For $\phi\in\SLn$ observe that
\begin{equation*}
(\tau_x^{-1}\circ\phi^{-1})(z)=\phi^{-1}z - x = \phi^{-1} (z-\phi x) = (\phi^{-1} \circ \tau_{\phi x}^{-1})(z)
\end{equation*}
for every $z\in\Rn$ and therefore
\begin{align*}
\begin{split}
\phi \oZ(f)+\oZ^0(f)\phi x &= \phi \oZ(f\circ \tau_x^{-1})\\
&=\oZ(f\circ\tau_x^{-1}\circ\phi^{-1})\\
&=\oZ(f\circ \phi^{-1} \circ \tau^{-1}_{\phi x})\\
&= \oZ(f\circ \phi^{-1}) + \oZ^0(f\circ \phi^{-1}) \phi x\\
&= \phi \oZ(f)+\oZ^0(f\circ \phi^{-1}) \phi x.
\end{split}
\end{align*}
Hence, $\oZ^0$ is $\SLn$ invariant. Moreover, for $s>0$ we have
\begin{multline*}
s^q \oZ(f)+s^q \oZ^0(f)x = s^q \oZ(f\circ \tau_x^{-1})=\oZ(s(f\circ\tau_x^{-1}))=\oZ((sf)\circ \tau_x^{-1}) = s^q \oZ(f) + \oZ^0(sf)x.
\end{multline*}
Lastly, if $f_k,f\in\LC$ are such that $\hlim_{k\to \infty} f_k = f$, then also $\hlim_{k\to\infty} f_k\circ \tau_x^{-1} = f\circ \tau_x^{-1}$. Hence, by the continuity of $\oZ$,
\begin{equation*}
\oZ(f_k) + \oZ^0(f_k)x = \oZ(f_k\circ\tau_x^{-1}) \longrightarrow \oZ(f\circ\tau_x^{-1}) = \oZ(f)+\oZ^0(f)x.
\end{equation*}
\end{proof}

\noindent
For the remainder of this section, let $n\geq 3$.

\begin{lemma}
\label{le:constants}
Let $\oZ:\LC\to\Kn$ be a continuous, equi-affinely covariant Minkowski valuation. There exist constants $c_1,c_2,d_1,d_2,d_4 \geq 0,\,c_3,d_3\in\R$ and $q\in\R$, with $q>0$ if $c_3\neq 0$, such that
\begin{align*}
\oZ(s\,e^{-\l_K}) \, &= \, s^q(d_1 K + d_2 (-K) + d_4 \vm(K) + d_3 \oM(K)),
\intertext{for every $K\in\Ko$ and $s>0$ and}
\oZ(s\,\Char_K) \, &= \, s^q(c_1 K + c_2 (-K) + c_3 \vm(K)),
\end{align*}
for every $K\in\Kn$ and $s> 0$. Furthermore,
\begin{equation*}
\oZ^0(f)=(c_1-c_2)V_0(f)^q+ c_3 V_n(f^q),
\end{equation*}
for every $f\in\LC$.
\end{lemma}
\begin{proof}
Since for $K,L\in\Ko$ such that $K\cup L\in\Ko$ we have
\begin{equation*}
\l_{K\cup L}=\l_K\mn \l_L,\qquad \l_{K\cap L}=\l_K\mx \l_L,
\end{equation*}
the map
\begin{equation}
\label{eq:k_to_y_lk}
K\mapsto \oZ(e^{-\l_K})
\end{equation}
defines a Minkowski valuation on $\Ko$. Furthermore, $\l_{\phi K}=\l_K\circ\phi^{-1}$ for every $\phi\in\SLn$ and $\l_{K_k}\eto \l_K$ for every sequence $K_j$ that converges to $K$ in $\Ko$ by Lemma~\ref{le:epi_lvl_sets}. Hence, (\ref{eq:k_to_y_lk}) defines a continuous, $\SLn$ covariant Minkowski valuation on $\Ko$. It follows from Theorem~\ref{thm:haberl_covariant} that there exist constants $d_1,d_2,d_4\geq 0$ and $d_3\in\R$ such that
\begin{equation*}
\oZ(s\, e^{-\l_K})=s^q \oZ(e^{-\l_K})=s^q(d_1 K + d_2(-K)+d_3\vm(K)+\d_4\oM(K))
\end{equation*}
for every $K\in\Ko$ and $s>0$, where $q\in\R$ denotes the degree of homogeneity of $\oZ$. Similarly, $K\mapsto \oZ(\Char_K)$ defines a continuous, $\SLn$ and translation covariant Minkowski valuation on $\Kn$. Hence, by Corollary~\ref{cor:haberl_covariant} there exist constants $c_1,c_2\geq 0$ and $c_3\in\R$ such that
\begin{equation*}
\oZ(s \Char_K) = s^q (c_1 K + c_2 (-K) + c_3 \vm(K)), 
\end{equation*}
for every $K\in\Kn$ and $s>0$.\par
For $K\in\Kn$, $x\in\Rn\backslash\{0\}$ and $s>0$ let $f:=s\Char_K\in\LC$ and observe that
\begin{align*}
\oZ(f)+\oZ^0(f)x &= \oZ(f\circ\tau_x^{-1})\\
&= \oZ(s\Char_{K+x})\\
&= s^q (c_1 K + c_2 (-K) + c_3 \vm(K) + (c_1-c_2 + c_3 V_n(K)) x)\\
&= \oZ(f) + s^q(c_1-c_2+c_3 V_n(K))x.
\end{align*}
On the other hand, by Lemma~\ref{le:z0_is_a_val} and Theorem~\ref{thm:class_inv}, there exist $\widetilde{c}_0,\widetilde{c}_n\in\R$ and $\widetilde{q}\in\R$, with $\widetilde{q}>0$ if $\widetilde{c}_n\neq 0$, such that
\begin{equation*}
\oZ^0(g)=\widetilde{c}_0 V_0(g)^{\widetilde q} + \widetilde{c}_n V_n(g^{\widetilde q}), 
\end{equation*} 
for every $g\in\LC$. Noting, that $V_0(f)^q=s^q$ and $V_n(f^q)=s^q V_n(K)$, a comparison shows that
\begin{equation*}
(c_1-c_2)s^q V_0(K)+c_3 s^q V_n(K) = \oZ^0(s\Char_K) = \widetilde{c}_0 s^{\widetilde q} V_0(K) + \widetilde{c}_n s^{\widetilde q} V_n(K),
\end{equation*}
for every $s>0$ and $K\in\Kn$. Choosing $K=\{0\}$ and $s=1$ gives $c_1-c_2 = \widetilde{c}_0$. With the same $K$ and arbitrary $s>0$ we have $q=\widetilde{q}$ and with any full-dimensional $K\in\Kn$ we obtain $\widetilde{c}_n=c_3$.
\end{proof}

\begin{lemma}
\label{le:c1_c2}
Let $\oZ:\LC\to\Kn$ be a continuous, equi-affinely covariant Minkowski valuation. If $c_1,c_2,d_1,d_2,q$ denote the constants from Lemma~\ref{le:constants}, then $c_1=q\,d_1$ and $c_2=q\,d_2$.
\end{lemma}
\begin{proof}
For $h>0$ let $u_h\in\CV$ be defined via $\epi u_h =  \epi \l_{[0,e_1/h]} \cap \{x_1 \leq 1\}$. Lemma~\ref{le:epi_lvl_sets} shows that $u_h \eto \Ind_{[0,e_1]}$ as $h\to 0$. Moreover, for $\l_h:=\l_{[0,e_1/h]}\circ \tau_{e_1}^{-1}+h$ we have
\begin{equation*}
u_h \mn \l_h =  \l_{[0,e_1/h]},\qquad u_h \mx \l_h = \Ind_{\{e_1\}}+h.
\end{equation*}
Similar as in the proof of Theorem~\ref{thm:class_inv}, let $\oY$ be the valuation on $\CV$ that corresponds to $\oZ$, that is $\oY(u)=\oZ(e^{-u})$ for every $u\in\CV$. Then $\oY$ is continuous, $\SLn$ and translation covariant and $\oY(u+t)=e^{-qt}\oY(u)$ for every $u\in\CV$ and $t\in\R$.
We have
\begin{equation*}
\oY(\l_h) = e^{-qh} \oY(\l_{[0,e_1/h]})+ (c_1-c_2)e^{-qh} e_1
\end{equation*}
and furthermore
\begin{align*}
c_1 = h(\oY(\Ind_{[0,e_1]}),e_1)&=\lim_{h\to 0^+} h(\oY(u_h),e_1)\\
&= \lim_{h\to 0^+} (h(\oY(\l_{[0,e_1/h]}),e_1)+h(\oY(\Ind_{\{e_1\}}+h),e_1)-h(\oY(\l_h),e_1))
\\
&=\lim_{h\to 0^+} (\tfrac{d_1}{h}+(c_1-c_2) e^{-qh} - e^{-qh} \tfrac{d_1}{h} - (c_1-c_2) e^{-qh})\\
&=\lim_{h\to 0^+} d_1 \tfrac{1-e^{-qh}}{h} =  q\,d_1.
\end{align*}
Similarly, evaluating the support functions at $-e_1$ shows that $c_2=q\,d_2$.
\end{proof}

\noindent
In the following we say that a Minkowski valuation $\oZ:\LC\to\Kn$ is \emph{trivial} if $\oZ(f)=\{0\}$ for every $f\in\LC$.\par
Using the earlier highlighted relation between valuations on $\LC$ and valuations on $\CV$, we obtain the following result from \cite[Lemma 8.7]{colesanti_ludwig_mussnig_mink}, where $\SLn$ covariant and translation invariant valuations on $\CV$ were studied.

\begin{lemma}
\label{le:trans_inv_mink_val}
Every continuous, $\SLn$ covariant and translation invariant Minkowski valuation $\oZ:\LC\to\Kn$ is uniquely determined by the values $\oZ(s\,\Char_K)$ with $s>0$ and $K\in\Kn$.
\end{lemma}

\begin{lemma}
\label{le:hom_of_pos_deg}
Every continuous, equi-affinely covariant Minkowski valuation $\oZ:\LC\to\Kn$ is either homogeneous of a positive degree or trivial.
\end{lemma}
\begin{proof}
Let $d_1,d_2,c_1,c_2,c_3$ and $q$ denote the constants from Lemma~\ref{le:constants} and suppose that $q\leq 0$. It follows from Lemma~\ref{le:constants} that $c_3=0$. Furthermore, since $c_1,c_2,d_1$ and $d_2$ are non-negative, Lemma~\ref{le:c1_c2} yields that also $c_1=c_2=0$. Hence, $\oZ^0\equiv 0$ and $\oZ$ is translation invariant. Moreover, $\oZ(s\,\Char_K)=\{0\}$ for every $s>0$ and $K\in\Kn$. Thus, Lemma~\ref{le:trans_inv_mink_val} shows that $\oZ$ is trivial.
\end{proof}

\begin{lemma}
\label{le:exp_lemma}
For $a,b\in\R$ and $q>0$ the following holds:
\begin{equation*}
\lim_{h\to 0^+} \left(a \frac{1-e^{-qh}}{h^2} - b \frac{e^{-qh}}{h}\right) = \begin{cases} +\infty \quad &\text{if } b<q\,a\\\tfrac{q}{2}b \quad &\text{if } b=q\,a\\ -\infty \quad &\text{if } b>q\,a.\end{cases}
\end{equation*}
\end{lemma}
\begin{proof}
Since,
\begin{equation*}
a \frac{1-e^{-qh}}{h^2} - b \frac{e^{-qh}}{h} = \frac{a\,(1-e^{-qh})-b\,h\,e^{-qh}}{h^2},
\end{equation*}
and
\begin{equation*}
\lim_{h\to 0^+}  \big( a\,(1-e^{-qh})-b \,h\,e^{-qh} \big) = 0,
\end{equation*}
we can apply L'Hospital's rule to obtain
\begin{align*}
\begin{split}
\lim_{h\to 0^+} \frac{a\,(1-e^{-qh})-b\,h e^{-qh}}{h^2} &= \lim_{h\to 0^+} \frac{q\,a\,e^{-q\,h} - b\,e^{-qh} + q\,b\,h e^{-qh}}{2h}\\
&= \lim_{h\to 0^+} \tfrac{e^{-qh}}{2h}(q\,a-b) + \tfrac{q}{2} b.
\end{split}
\end{align*}
The claim now follows since $\tfrac{e^{-qh}}{2h} \to +\infty$ as $h\to 0^+$.
\end{proof}

\begin{lemma}
\label{le:c3_d4}
Let $\oZ:\LC\to\Kn$ be a continuous, equi-affinely covariant Minkowski valuation. If $c_3,d_3,d_4, q$ denote the constants from Lemma \ref{le:constants}, then $c_3= \tfrac{q^{n+1}}{(n+1)!}d_3$ and $d_4 = 0$.
\end{lemma}
\begin{proof}
By Lemma~\ref{le:hom_of_pos_deg}, we can assume without loss of generality that $q>0$. Define\break$v\in\CV$ via
\begin{equation*}
\{v < 0\}=\emptyset,\quad \{v\leq s\} = [0,e_1] + \conv\{0,s\,e_2,\ldots,s\,e_n\},
\end{equation*}
for every $s\geq 0$. Now, for $h>0$ let $T_{1/h}$ be defined as in Lemmas~\ref{le:t_l_h}~\&~\ref{le:vn_l_t_lambda} and define the function $u_h$ via
\begin{equation*}
\{u_h\leq s\} = \{\l_{T_{1/h}} \leq s\} \cap \{x_1 \leq 1\},
\end{equation*}
for every $s\in\R$. It is easy to see that $u_h\in\CV$ and furthermore,
\begin{eqnarray*}
\{u_h \leq s\} \cup \{\l_{T_{1/h}}\circ \tau_{e_1}^{-1} + h \leq s \} &=& \{\l_{T_{1/h}}\leq s\}\\
\{u_h \leq s\} \cap \{\l_{T_{1/h}}\circ \tau_{e_1}^{-1} + h \leq s \} &=& \{\l_{\conv\{0,e_2,\ldots,e_n\}} \circ \tau_{e_1}^{-1} +h \leq s\},
\end{eqnarray*}
for every $s\in\R$. Thus, denoting $\oY(u)=\oZ(e^{-u})$ for $u\in\CV$, this gives
\begin{equation}
\label{eq:val_t_l_h}
\oY(u_h)+\oY(\l_{T_{1/h}}\circ\tau_{e_1}^{-1}+h) = \oY(\l_{T_{1/h}}) + \oY(\l_{\conv\{0,e_2,\ldots,e_n\}}\circ \tau_{e_1}^{-1}+h).
\end{equation}
By Lemmas~\ref{le:vn_l_t_lambda}~\&~\ref{le:constants} we have
\begin{eqnarray*}
\oY(\l_{T_{1/h}}\circ\tau_{e_1}^{-1}+h) &=& e^{-qh} \oY(\l_{T_{1/h}})+e^{-qh} ((c_1-c_2)+\tfrac{c_3}{h\, q^n}) e_1\\
\oY(\l_{\conv\{0,e_2,\ldots,e_n\}}\circ \tau_{e_1}^{-1}+h) &=& e^{-qh} \oY(\l_{\conv\{0,e_2,\ldots,e_n\}})+e^{-qh}(c_1-c_2) e_1.
\end{eqnarray*}
Furthermore, using Lemma~\ref{le:t_l_h} we obtain for the support functions
\begin{eqnarray*}
h(\oY(\l_{T_{1/h}}),e_1) &=& \tfrac{d_1}{h} + \tfrac{d_3+d_4}{h^2\, (n+1)!},\\
h(\oY(\l_{T_{1/h}}\circ\tau_{e_1}^{-1}+h),e_1) &=& e^{-qh} \big(\tfrac{d_1}{h}+ \tfrac{d_3+d_4}{h^2\, (n+1)!}  +(c_1-c_2)  + \tfrac{c_3}{h\, q^n}\big),\\
h(\oY(\l_{\conv\{0,e_2,\ldots,e_n\}}),e_1) &=& e^{-qh} (c_1-c_2).
\end{eqnarray*}
Observe, that for $h\to 0^+$ we have $u_h\eto v$. Hence, by the continuity of $\oY$ and (\ref{eq:val_t_l_h}), we have
\begin{align*}
h(\oY(v),e_1)&=\lim_{h\to 0^+} h(\oY(u_h),e_1)\\
&=\lim_{h\to 0^+} \big(\tfrac{d_1}{h} (1-e^{-qh}) +  \tfrac{d_3+d_4}{h^2\, (n+1)!} (1-e^{-qh}) - \tfrac{c_3}{h\, q^n} e^{-qh}\big)\\
&= q d_1 + \lim_{h\to 0^+} \big(\tfrac{d_3+d_4}{(n+1)!} \tfrac{1-e^{-qh}}{h^2} - \tfrac{c_3}{q^n} \tfrac{e^{-qh}}{h} \big).
\end{align*}
Since this expression must be finite, it follows from Lemma~\ref{le:exp_lemma} that
\begin{equation*}
\tfrac{c_3}{q^n} = q \tfrac{d_3+d_4}{(n+1)!}.
\end{equation*}
Similarly, repeating the calculations above but evaluating the support functions at $-e_1$ gives
\begin{equation*}
\tfrac{c_3}{q^n} = q \tfrac{d_3-d_4}{(n+1)!}.
\end{equation*}
Hence, $d_4=0$ and $c_3=\tfrac{q^{n+1}}{(n+1)!} d_3$.
\end{proof}

\noindent
By Lemma~\ref{le:reduction}, every continuous, equi-affinely covariant Minkowski valuation $\oZ$ on $\LC$ is uniquely determined by the constants $c_1,c_2,c_3,d_1,d_2,d_3,d_4$ and $q$ from Lemma~\ref{le:constants}. By Lemmas~\ref{le:c1_c2}~\&~\ref{le:c3_d4} we have $d_1 = \tfrac{c_1}{q}, d_2=\tfrac{c_2}{q}, d_3=\tfrac{(n+1)!}{q^{n+1}} c_3$ and $d_4=0$. Hence, $\oZ$ is completely determined by the constants $c_1,c_2,c_3$ and $q$. Thus, we have the following result.
\begin{lemma}
\label{le:z_determined_on_s_char}
Every continuous, equi-affinely covariant Minkowski valuation $\oZ:\LC\to\Kn$ is uniquely determined by the values $\oZ(s\,\Char_K)$ with $s>0$ and $K\in\Kn$.
\end{lemma}

\subsection{Proof of Theorem~\ref{thm:class_cov}}
By Lemmas~\ref{le:id_on_lc_is_a_val}~\&~\ref{le:mv_on_lc_is_a_val}, the operator
\begin{equation*}
f \mapsto c_1 [f^q]+c_2(-[f^q])+c_3 \vm(f^q),
\end{equation*}
defines a continuous, equi-affinely covariant Minkowski valuation on $\LC$ for every $c_1,c_2\geq 0$, $c_3\in\R$ and $q>0$.\par
Conversely, let $\oZ:\LC\to\R$ be a continuous, equi-affinely covariant Minkowski valuation. For arbitrary $K\in\Kn$ and $s>0$, let $f=s\,\Char_K$. By Lemma~\ref{le:constants}, there exist constants $c_1,c_2\geq 0$ and $c_3,q\in\R$ such that
\begin{equation*}
\oZ(f)=s^q(c_1 K + c_2 (-K) + c_3 \vm(K))
\end{equation*}
and by Lemma~\ref{le:hom_of_pos_deg} we may assume that $q>0$. Since
\begin{eqnarray*}
h([f^q],z)&=&\int_0^{+\infty} h(\{s^q \Char_K \geq t\},z)\d t = s^q\,h(K,z)\\
h(\vm(f^q),z) &=& \int_{\Rn} s^q \Char_K(x) (x\cdot z)\d x = s^q\,h(\vm(K),z)
\end{eqnarray*}
we have $\oZ(f)=c_1 [f^q]+c_2(-[f^q]) + c_3 \vm(f^q)$. Thus, Lemma~\ref{le:z_determined_on_s_char} completes the proof of the theorem.
\subsection*{Acknowledgments}
The author was supported, in part, by Austrian Science Fund (FWF) Project P25515-N25.

\small


\begin{thebibliography}{10}


\bibitem{Alesker99}
\textsc{Alesker,~S.} {\em Continuous rotation invariant valuations on convex sets}, Ann.
  of Math. (2) {\bf 149} (1999), 977--1005.

\bibitem{Alesker01}
\textsc{Alesker,~S.} {\em Description of translation invariant valuations on convex sets
  with solution of {P}.\ {M}c{M}ullen's conjecture}, Geom. Funct. Anal. {\bf 11}
  (2001), 244--272.

\bibitem{Alesker_convex}
\textsc{Alesker,~S.} {\em Valuations on convex functions and convex
sets and Monge-Amp\`ere operators}, Preprint (arXiv:1703.08778).

\bibitem{baryshnikov_ghrist_wright}
\textsc{Baryshnikov,~Y., Ghrist,~R. and Wright,~M.} {\em Hadwiger's {T}heorem for
  definable functions}, Adv. Math. {\bf 245} (2013), 573--586.

\bibitem{blaschke}
\textsc{Blaschke,~W.} {\em Vorlesungen \"uber {I}ntegralgeometrie. H.2.}, Teubner, Berlin, 1937.

\bibitem{bobkov_colesanti_fragala}
\textsc{Bobkov,~S.~G., Colesanti,~A., and Fragal{\`a},~I.} {\em Quermassintegrals of quasi-concave functions and generalized {P}r\'ekopa-{L}eindler inequalities},
  Manuscripta Math. {\bf 143} (2014), 131--169. 

\bibitem{cavallina_colesanti}
\textsc{Cavallina,~L. and Colesanti,~A.} {\em Monotone valuations on the space of convex functions}, Anal. Geom. Metr. Spaces {\bf 3} (2015), 167--211.

\bibitem{colesanti_fragala}
\textsc{Colesanti,~A. and Fragal{\`a},~I.} {\em The first variation of the total mass of log-concave functions and related inequalities}, Adv. Math. {\bf 244} (2013),
  708--749.

\bibitem{colesanti_lombardi}
\textsc{Colesanti,~A. and Lombardi,~N.} {\em Valuations on the space of quasi-concave
  functions}, Geo\-metric Aspects of Functional Analysis: Israel Seminar (GAFA) 2014--2016 (B. Klartag and E. Milman, eds.), Lecture Notes in Mathematics 2169, Springer International Publishing, Cham, 2017, 71--105.

\bibitem{colesanti_lombardi_parapatits}
\textsc{Colesanti,~A., Lombardi,~N. and Parapatits,~L.} {\em Translation invariant valuations on quasi-concave functions}, Preprint (arXiv:1703.06867).

\bibitem{colesanti_ludwig_mussnig}
\textsc{Colesanti,~A., Ludwig,~M. and Mussnig,~F.} {\em Valuations on convex functions}, Int. Math. Res. Not. IMRN, in press (arXiv:1703.06455).

\bibitem{colesanti_ludwig_mussnig_mink}
\textsc{Colesanti,~A., Ludwig,~M. and Mussnig,~F.} {\em Minkowski valuations on convex functions}, Preprint (arXiv:1707.05242).

\bibitem{haberl}
\textsc{Haberl,~C.} {\em Minkowski valuations intertwining with the special linear
  group}, J. Eur. Math. Soc. (JEMS) {\bf 14} (2012), 1565--1597.

\bibitem{haberl_parapatits_centro}
\textsc{Haberl,~C. and Parapatits,~L.} {\em The centro-affine {H}adwiger theorem}, J.
  Amer. Math. Soc. {\bf 27} (2014), 685--705.

\bibitem{haberl_parapatits_moments}
\textsc{Haberl,~C. and Parapatits,~L.} {\em Moments and valuations}, Amer. J. Math. {\bf 138} (2017), 1575--1603.

\bibitem{hadwiger1952}
\textsc{Hadwiger,~H.} {\em Translationsinvariante, additive und schwachstetige {P}olyederfunktionale}, Arch. Math. {\bf 3} (1952), 387--394.

\bibitem{hadwiger}
\textsc{Hadwiger,~H.} {\em {V}or\-lesungen \"uber {I}nhalt, {O}ber\-fl\"ache und
  {I}so\-peri\-metrie}, Springer, Berlin, 1957.

\bibitem{hadwiger_schneider}
\textsc{Hadwiger,~H. and Schneider,~R.} {\em Vektorielle {I}ntegralgeometrie}, Elem. Math. {\bf 26} (1971), 49--57.

\bibitem{huang_lyz_acta}
\textsc{Huang,~Y., Lutwak,~E., Yang,~D., and Zhang,~G.} {\em Geometric measures in the dual {B}runn--{M}inkowski theory and their associated {M}inkowski problems}, Acta Math. {\bf 216} (2016), 325--388.

\bibitem{leindler}
\textsc{Leindler,~L.} {\em On a certain converse of {H\"o}lder's inequality {II}}, Acta Sci. Math. (Szeged) {\bf 33} (1972), 217--223.

\bibitem{li_leng}
\textsc{Li,~J., and Leng,~G.} {\em {$L_p$} {M}inkowski valuations on polytopes}, Adv. Math. {\bf 299} (2016), 139--173.

\bibitem{li_ma}
\textsc{Li,~J. and Ma,~D.} {\em Laplace transforms and valuations}, J. Funct. Anal. {\bf 272} (2017), 738--758.

\bibitem{ludwig_projection}
\textsc{Ludwig,~M.} {\em Projection bodies and valuations}, Adv. Math. {\bf 172} (2002),
  158--168.

\bibitem{ludwig_minkowski}
\textsc{Ludwig,~M.} {\em Minkowski valuations}, Trans. Amer. Math. Soc. {\bf 357}
  (2005), 4191--4213.

\bibitem{ludwig_fisher}
\textsc{Ludwig,~M.} {\em Fisher information and matrix-valued valuations}, Adv. Math. {\bf 226}
  (2011), 2700--2711.

\bibitem{ludwig_sobolev}
\textsc{Ludwig,~M.} {\em Valuations on {S}obolev spaces}, Amer. J. Math. {\bf 134} (2012), 827--842.

\bibitem{ludwig_covariance}
\textsc{Ludwig,~M.} {\em Covariance matrices and valuations}, Adv. in Appl. Math. {\bf
  51} (2013), 359--366.

\bibitem{ludwig_reitzner_annals}
\textsc{Ludwig,~M. and Reitzner,~M.} {\em A classification of \,{${\rm SL}(n)$\!} invariant
  valuations}, Ann. of Math. (2) {\bf 172} (2010), 1219--1267.

\bibitem{ludwig_reitzner}
\textsc{Ludwig,~M. and Reitzner,~M.} {\em {${\rm SL}(n)$\!} invariant
  valuations on polytopes}, Discrete Comput. Geom. {\bf 57} (2017), 571--581.

\bibitem{ma}
\textsc{Ma,~D.} {\em Real-valued valuations on {S}obolev spaces}, Sci. China Math. {\bf
  59} (2016), 921--934.

\bibitem{mcmullen_weakly_cont_val}
\textsc{McMullen,~P.} {\em Weakly continuous valuations on convex polytopes}, Arch. Math. {\bf 41} (1983), 555--564.

\bibitem{mcmullen_handbook}
\textsc{McMullen,~P.} {\em Valuations and dissections}, Handbook of {C}onvex {G}eometry, Vol. B (P.M. Gruber and J.M. Wills, eds.), North-Holland, Amsterdam, 1993, 933–-990.


\bibitem{mcmullen_schneider}
\textsc{McMullen,~P. and Schneider,~R.} {\em Valuations on convex bodies}, Convexity and its Applications (P.M. Gruber and J.M. Wills, eds.), Birkh{\"a}user, Basel, 1983, 170-–247.

\bibitem{milman_rotem}
\textsc{Milman,~V. and Rotem,~L.} {\em Mixed integrals and related inequalities}, J. Funct. Anal. {\bf 264} (2013), 570--604.

\bibitem{ober_minkowski}
\textsc{Ober,~M.} {\em {$L\sb p$}-{M}inkowski valuations on {$L\sp q$}-spaces}, J. Math.
  Anal. Appl. {\bf 414} (2014), 68--87.

\bibitem{prekopa}
\textsc{Pr{\'e}kopa,~A.} {\em Logarithmic concave measures with application to stochastic programming}, Acta Sci. Math. (Szeged) {\bf 32} (1971), 301--315.

\bibitem{rockafellar_wets}
\textsc{Rockafellar,~R.~T. and Wets,~R.~J.-B.} {\em Variational {A}nalysis}, Grundlehren der Mathe\-matischen Wissen\-schaften, vol. 317, Springer-Verlag, Berlin, 1998.

\bibitem{schneider_steiner_point}
\textsc{Schneider,~R.} {\em On {S}teiner points of convex bodies}, Israel J. Math. {\bf 9} (1971), 241--249.

\bibitem{schneider}
\textsc{Schneider,~R.} {\em Convex {B}odies: the {B}runn-{M}inkowski {T}heory}, {S}econd
  expanded ed., Encyclopedia of Mathe\-matics and its Applications, vol. 151,
  Cambridge University Press, Cambridge, 2014.

\bibitem{schuster_wannerer_gln_contra_mink}
\textsc{Schuster,~F.~E. and Wannerer,~T.} {\em {${\rm GL}(n)$\!} contravariant {M}inkowski
  valuations}, Trans. Amer. Math. Soc. {\bf 364} (2012), 815--826.

\bibitem{tsang_val_on_lp}
\textsc{Tsang,~A.} {\em Valuations on ${L}^p$ spaces}, Int. Math. Res. Not. {\bf 20}
  (2010), 3993--4023.

\bibitem{tsang_mink_val_on_lp}
\textsc{Tsang,~A.} {\em Minkowski valuations on ${L}^p$-spaces}, Trans. Amer. Math. Soc.
  {\bf 364} (2012), 6159--6186.

\bibitem{wannerer_gln_equi}
\textsc{Wannerer,~T.} {\em {${\rm GL}(n)$\!} equivariant {M}inkowski valuations}, Indiana
  Univ. Math. J. {\bf 60} (2011), 1655--1672.

\end{thebibliography}

%
%
%
%
%
%
%
%
%

\bigskip\bigskip\bigskip\footnotesize
\parindent 0pt

\parbox[t]{8.5cm}{
Fabian Mussnig\\
Institut f\"ur Diskrete Mathematik und Geometrie\\
Technische Universit\"at Wien\\
Wiedner Hauptstra\ss e 8-10/1046\\
1040 Wien, Austria\\
e-mail: fabian.mussnig@alumni.tuwien.ac.at
}

\end{document}